\documentclass[oneside,reqno,11pt]{amsart}
\usepackage{amsmath,amsfonts,amscd,amssymb,enumitem,
extarrows,latexsym}

\usepackage{hyperref}
\usepackage{soul}
\usepackage{mathrsfs}

\usepackage[nobysame]{amsrefs}
\usepackage{comment}
\usepackage{tikz}
\usepackage{marginnote}
\usepackage{anysize}
\usepackage{bbm}
\usepackage{cancel}

\usepackage[normalem]{ulem}

\numberwithin{equation}{section}

\newtheorem{theorem}{Theorem}[section]
\newtheorem{lemma}[theorem]{Lemma}
\newtheorem{proposition}[theorem]{Proposition}
\newtheorem{corollary}[theorem]{Corollary}

\theoremstyle{definition}
\newtheorem{definition}[theorem]{Definition}

\theoremstyle{remark}
\newtheorem{remark}[theorem]{Remark}

\numberwithin{equation}{section}

\newcommand{\ind}{\operatorname{ind}}
\newcommand{\id}{\operatorname{id}}

\newcommand{\dm}{\partial M}
\newcommand{\CC}{\mathbb{C}}

\newcommand{\RR}{\mathbb{R}}

\newcommand{\ZZ}{\mathbb{Z}}

\newcommand{\AAA}{\mathcal{A}}

\newcommand{\DD}{\mathcal{D}}

\newcommand{\cHH}{\check{H}}
\newcommand{\hHH}{\hat{H}}

\newcommand{\bfu}{\mathbf{u}}
\newcommand{\bfv}{\mathbf{v}}
\newcommand{\upper}{\uppercase\expandafter}

\newcommand{\n}{\nabla}
\newcommand{\p}{\partial}
\newcommand{\pM}{{\p M}}

\newcommand{\codim}{\operatorname{codim}}

\newcommand{\End}{\operatorname{End}}
\newcommand{\range}{\operatorname{range}}
\newcommand{\ad}{{\rm ad}}

\newcommand{\tAAA}{{\AAA}}
\newcommand{\E}{\tilde{E}_\pM}

\begin{document}

\author{Maxim Braverman${}^\dag$}
\address{Department of Mathematics,
Northeastern University,
Boston, MA 02115,
USA}

\email{maximbraverman@neu.edu}
\urladdr{www.math.neu.edu/~braverman/}

\author{Pengshuai Shi}
\address{Beijing International Center for Mathematical Research (BICMR),
Beijing (Peking) University,
Beijing 100871,
China}

\email{pengshuai.shi@gmail.com}

\thanks{${}^\dag$†Partially supported by the Simons Foundation collaboration grant \#G00005104.}

\title[Index of a local boundary value problem]{The index of a local boundary value problem for strongly Callias-type operators}

\begin{abstract}
We consider a complete Riemannian manifold $M$ whose boundary is a  disjoint union of finitely many complete  connected Riemannian manifolds. We compute the index of a local boundary value problem for a strongly Callias-type operator on $M$. Our result extends an index theorem of D.~Freed to non-compact manifolds, thus providing a new insight on the Ho\v{r}ava-Witten anomaly. 
\end{abstract}

\maketitle

\setcounter{tocdepth}{1}

\section{Introduction}\label{S:introduction}

Mathematical description of many anomalies in quantum field theory is given by  index theorems for a boundary value problems, cf. \cite{Witten85GravAnom,AtSinger84,Freed86,BarStrohmaier16} (see \cite[Ch.~11]{Bertlmann01} for more details). A new type of anomalies, related to index computation on an {\em odd dimensional} space $\RR^{10}\times [0,1]$, was discovered by Ho\v{r}ava and Witten in \cite{HoravaWitten96}. 
Freed, \cite{Freed98}, replaced $\RR^{10}$ with a compact manifold, and proved a new index theorem for a {\em local boundary value problem} on a compact manifold with boundary, which explains the Ho\v{r}ava-Witten anomaly, but only for this compact case. It is desirable to give a mathematically rigorous construction of such an index on non-compact manifolds with non-compact boundary. 

A systematic treatment of boundary value problems for strongly Callias-type operators on non-compact manifolds with non-compact boundary was given in \cite{BrShi17}, where we extended  the approach of \cite{BarBallmann12} to non-compact setting. In particular, we defined elliptic boundary conditions and proved that the corresponding boundary value problem is Fredholm. We emphasize that these results are valid for so called  strongly Callias-type operators --- Dirac operators coupled with an electric potential satisfying certain growth conditions at infinity. 

One advantage of the approach to boundary value problems in \cite{BarBallmann12,BrShi17} is that it unifies local and non-local (eg. Atiyah--Patodi--Singer) boundary conditions. In \cite{BrShi17,BrShi17b} we studied the index defined by (generalized) Atiyah--Patodi--Singer boundary conditions on manifolds with non-compact boundary. In \cite{Shi18} the second author studied the Calder\'on projection in the non-compact situation and obtained an expression of the index of an elliptic boundary value problem in terms of the relative index of the Calder\'on projection and the projection onto the boundary conditions. In the current paper we study local boundary conditions for Callias-type operators and prove a non-compact analogue of Freed's index theorem \cite[Theorem~B]{Freed98}.

We now give a brief description of our results.

Let $M$ be a complete Riemannian manifold with non-compact boundary $\pM$. We assume that $\pM=\bigsqcup_{j=1}^k N_j$  is a disjoint union of finitely many connected components. Then each $N_j$ is a complete manifold without boundary. Let $E$ be an (ungraded) Dirac bundle over $M$ and let $D$ denote the Dirac operator on $E$. Let $\DD=D+\Psi$ be a formally self-adjoint Callias-type operator on $M$. We impose slightly stronger conditions on the growth of the potential $\Psi$ and call the operators satisfying these conditions {\em strongly Callias-type}. On manifolds without boundary these conditions guarantee that $\DD$ has a discrete spectrum. 

The restriction $\AAA_j$ ($j=0,\ldots,k$) of $\DD$ to the boundary component $N_j$ is a self-adjoint strongly Callias-type operator on $N_j$ and, hence, has discrete spectrum. In particular, it is Fredholm. Moreover, the Clifford multiplication by the unit normal vector to the boundary defines a grading on $E_{N_j}:= E|_{N_j}$ and $\AAA_j$ is odd with respect to this grading. We denote by $\ind\AAA_j$ its index. 

Set $B_j^\pm:= L^2(N_j,E_{N_j}^\pm)$. Let $\epsilon:= (\epsilon_1,\ldots,\epsilon_k)$ where $\epsilon_j=\pm$ and set $B^\epsilon:= \bigoplus_{j=1}^kB_j^{\epsilon_j}$. Then $B^\epsilon\subset L^2(\pM,E_{\pM})$, where $E_{\pM}:=E|_{\pM}$. We use the result in Section \ref{S:ellbc} to show that $B^\epsilon$ defines an elliptic boundary condition for $\DD$. Hence the operator 
\begin{equation}\label{E:IDB}
	\DD_{B^\epsilon}:\, 
	\big\{u\in L^2(M,E):\, u|_\pM\in B^\epsilon\,\big\}
	\ \to \ 
	L^2(M,E). 
\end{equation}
is Fredholm. 

Our main result in this paper (cf. Theorem~\ref{T:mixindfor}) is the following generalization of of \cite[Theorem~B]{Freed98}:
\begin{equation}\label{E:IindBepsilon}
\ind\DD_{B^\epsilon}\;=\;\sum_{\{j\,:\,\epsilon_j=+\}}\ind\AAA_j\;=\;-\sum_{\{j\,:\,\epsilon_j=-\}}\ind\AAA_j.
\end{equation}

If all $\epsilon_j$ have the same sign, then \eqref{E:IindBepsilon} becomes 
\[
	\ind\DD_{B^\epsilon}\ = \ \sum_{j=1}^k\, \ind \AAA_j\ = \ 0.
\] 
This gives a new proof of the cobordism invariance of the index of Callias-type operators, cf. \cite{BrShi16}. 

Freed, \cite{Freed98}, only considers the case when the dimension of $M$ is odd. This is because the index of an elliptic differential operator on a {\em compact} odd-dimensional manifold without boundary vanishes and the compact analogue of \eqref{E:IindBepsilon}  is trivial when $\dim M=\rm{even}$. This is not the case for the index of  Callias-type operators on non-compact manifolds. In fact, the odd-dimensional case is very interesting and is the subject of the  celebrated Callias-type index theorem, \cite{Anghel93Callias,Bunke95}. That is why we don't assume that the dimension of $M$ is odd.

\section{Boundary value problems for manifolds with non-compact boundary}\label{S:index}

In the beginning of this section we briefly recall the notion of strongly Callias-type operator and define the scale of Sobolev spaces defined by such an operator. We then recall the definition of an elliptic boundary value problem for a strongly Callias-type operator \cite{BrShi17} and define its index.  

\subsection{Strongly Callias-type operator}\label{SS:strCallias}
Let $M$ be a complete Riemannian manifold (possibly with boundary) and let $E\to M$ be a Dirac bundle over $M$, cf. \cite[Definition~II.5.2]{LawMic89}. In particular, $E$ is a Hermitian vector bundle endowed with a Clifford multiplication $c:T^*M\to \End(E)$ and a compatible Hermitian connection $\n^E$.  Let $D:C^\infty(M,E)\to C^\infty(M,E)$ be the Dirac operator defined by the connection $\n^E$. Let $\Psi\in{\rm End}(E)$ be a self-adjoint bundle map (called a \emph{Callias potential}). Then
\[
	\DD\;:=\;D\,+\,\Psi
\]
is a formally self-adjoint Dirac-type operator on $E$ and
\begin{equation}\label{E:Calliaseq}
	\DD^2\;=\;D^2+\Psi^2+[D,\Psi]_+,
\end{equation}
where $[D,\Psi]_+:=D\circ\Psi+\Psi\circ D$ is the anticommutator of the operators $D$ and $\Psi$.

\begin{definition}\label{D:saCallias}
We call $\DD$ a \emph{self-adjoint strongly Callias-type operator} if
\begin{enumerate}
\item $[D,\Psi]_+$ is a zeroth order differential operator, i.e. a bundle map;
\item for any $R>0$, there exists a compact subset $K_R\subset M$ such that
\begin{equation}\label{E:strinvinfA}
	\Psi^2(x)\;-\;\big|[D,\Psi]_+(x)\big|\;\ge\;R
\end{equation}
for all $x\in M\setminus K_R$.  In this case, the compact set $K_R$ is called an \emph{R-essential support} of $\DD$.
\end{enumerate}
\end{definition}

\subsection{Restriction to the boundary}\label{SS:restriction}
Assume that the Riemannian metric $g^M$ is {\em product near the boundary}, that is, there exists a neighborhood $U\subset M$ of the boundary which is isometric to the cylinder
\begin{equation}\label{E:Zr}
	Z_r\ := \ [0,r)\times\dm. 
\end{equation} 
In the following we identify $U$ with $Z_r$ and denote by $t$ the coordinate along the axis of $[0,r)$. Then the inward unit normal to the boundary is given by $\tau = dt$. 

Furthermore, we assume that the Dirac bundle $E$ is {\em product near the boundary}. This means that the Clifford multiplication $c:T^*M\to{\rm End}(E)$ and the connection $\nabla^E$ have product structure on $Z_r$, cf. \cite[\S3.7]{BrShi17}. Then the restriction of $D$ to $Z_r$ takes the form 
\begin{equation}\label{E:productD}
	D\; = \;c(\tau)(\p_t+A),
\end{equation}
where
\(
	A:\;C^\infty(\dm,E_{\dm})\;\to\; C^\infty(\dm,E_{\dm})
\)
is a self-adjoint Dirac-type operator which anticommutes with $c(\tau)$:
\begin{equation}\label{E:[A,c]}
	c(\tau)\circ A\ = \ -\,A\circ c(\tau).
\end{equation}

Let $\DD=D+\Psi$ be a self-adjoint strongly Callias-type operator. Suppose $\Psi$ does not depend on $t$ on $Z_r$. Then the restriction of $\DD$ to $Z_r$ is given by 
\begin{equation}\label{E:productDD}
	\DD\; = \;c(\tau)(\p_t+\AAA),
\end{equation}
where
\(
		\AAA \;:= \;A-c(\tau)\Psi:\,C^\infty(\dm,E_{\dm})
		\;\to \;C^\infty(\dm,E_{\dm})
\)
is the \emph{restriction of $\DD$ to the boundary}.

Condition (i) of\/ Definition~\ref{D:saCallias} is equivalent to the condition that $\Psi$ anticommutes with the Clifford multiplication: $\big[c(\xi),\Psi\big]_+ = 0$, for all $\xi\in T^*M$. 
It follows that  $c(\tau)\Psi\in{\rm End}(E_{\dm})$ is a self-adjoint bundle map which anticommutes with $c(\tau)$. Hence, using \eqref{E:[A,c]}, we obtain 
\begin{equation}\label{E:cAAA}
 	c(\tau)\circ \AAA\ = \ -\,\AAA\circ c(\tau).
\end{equation} 
In addition,  $\AAA$ is a strongly Callias-type operator, cf. Lemma~3.12 of \cite{BrShi17}.  In particular, it has discrete spectrum. 

\begin{definition}\label{D:producDD}
We say that a self-adjoint strongly Callias-type operator $\DD$ is {\em product near the boundary} if the Dirac bundle $E$ is  product near the boundary and the restriction of the Callias potential $\Psi$ to $Z_r$ does not depend on $t$. The operator $\AAA$ of \eqref{E:productDD} is called the {\em restriction of\/ $\DD$ to the boundary}. 
\end{definition}

\subsection{The grading on the boundary}\label{SS:grading}

Let $E_{\dm}^\pm\subset E_{\dm}$ denote the span of the eigenvectors of $c(\tau)$ associated with eigenvalue $\pm i$. Then $E_{\dm}= E_{\dm}^+\oplus E_{\dm}^-$. By \eqref{E:cAAA}, with respect to this decomposition 
\begin{equation}\label{E:AAAgrading}
	\AAA\ = \ \begin{pmatrix}
	0&\AAA^-\\
	\AAA^+&0
	\end{pmatrix},
\end{equation}
where $\AAA^\pm:= \AAA|_{E_{\dm}^\pm}$.

\subsection{Sobolev spaces}\label{SS:Sobolev}
We recall the definition of Sobolev spaces $H^s_\AAA(\pM,E_\pM)$ of sections over $\pM$ which depend on the boundary operator $\AAA$, cf. \cite[\S3.13]{BrShi17}.

\begin{definition}\label{D:Sobolev}
Set
\[
	C_\AAA^\infty(\dm,E_{\dm})	\;:=\;
	\Big\{\,\bfu\in C^\infty(\dm,E_{\dm}):\,
	\big\|(\id+\AAA^2)^{s/2}\bfu\big\|_{L^2(\dm,E_\pM)}^2<+\infty\mbox{ for all }s\in\RR\,\Big\}.
\]
For all $s\in\RR$ we define the \emph{Sobolev $H_\AAA^s$-norm} on $C_\AAA^\infty(\dm,E_{\dm})$ by
\begin{equation}\label{E:Sobnorm}
	\|\bfu\|_{H_\AAA^s(\dm,E_\pM)}^2\;:=\;
		\big\|(\id+\AAA^2)^{s/2}\bfu\big\|_{L^2(\dm,E_\pM)}^2.
\end{equation}
The Sobolev space $H_\AAA^s(\dm,E_{\dm})$ is defined to be the completion of $C_\AAA^\infty(\dm,E_{\dm})$ with respect to this norm.
\end{definition}

\subsection{The hybrid Soblev spaces}\label{SS:hybrid}
For $I\subset\RR$, let $P_I^\AAA\;:L^2(\pM,E_\pM)\to L^2(\pM,E_\pM)$  
be the spectral projection onto the span of the eigenvectors of $\AAA$ with eigenvalues in $I$. It's easy to see that $P_I$ extends to a continuous projection on $H^s_\AAA(\pM,E_\pM)$ for all $s\in \RR$. We set 
\[
	H_I^s(\AAA)\;:=\;
	P_I^\AAA(H_\AAA^s(\dm,E_{\dm}))\;\subset\;H_\AAA^s(\dm,E_{\dm}).
\]

\begin{definition}\label{D:hybrid space}
For $a\in\RR$, we define the \emph{hybrid} Sobolev space
\begin{equation}\label{E:checkH}
\begin{aligned}
	\cHH(\AAA)\;&:=\;
	H_{(-\infty,a)}^{1/2}(\AAA)\,\oplus\,H_{[a,\infty)}^{-1/2}(\AAA)\;\subset \;H^{-1/2}_\AAA(\dm,E_{\dm}), \\
	\hHH(\AAA)\;&:=\;
	H_{(-\infty,a)}^{-1/2}(\AAA)\,\oplus\,H_{[a,\infty)}^{1/2}(\AAA)\;\subset \;H^{-1/2}_\AAA(\dm,E_{\dm}) \\
\end{aligned}
\end{equation}
with $\cHH$-norm or $\hHH$-norm
\[
\begin{aligned}
	\|\bfu\|_{\cHH(\AAA)}^2\;&:=\;
	\big\|P_{(-\infty,a)}^\AAA\bfu\big\|_{H_\AAA^{1/2}(\dm,E_\pM)}^2\;+\;
	\big\|P_{[a,\infty)}^\AAA\bfu\big\|_{H_\AAA^{-1/2}(\dm,E_\pM)}^2, \\
	\|\bfu\|_{\hHH(\AAA)}^2\;&:=\;
	\big\|P_{(-\infty,a)}^\AAA\bfu\big\|_{H_\AAA^{-1/2}(\dm,E_\pM)}^2\;+\;
	\big\|P_{[a,\infty)}^\AAA\bfu\big\|_{H_\AAA^{1/2}(\dm,E_\pM)}^2.
\end{aligned}
\]
\end{definition}

The space $\cHH(\AAA)$ or $\hHH(\AAA)$ is independent of the choice of $a$. They are dual to each other. Note from \eqref{E:cAAA} that $c(\tau)$ induces an isomorphism between $\cHH(\AAA)$ and $\hHH(\AAA)$. By Theorem~3.39 of \cite{BrShi17}, the hybrid space $\cHH(\AAA)$ coincides with the space of restriction to the boundary of a section of $E$ which lies in the maximal domain of $\DD$.

\subsection{Elliptic boundary value problems}\label{SS:boundaryvalue}

We are now ready to define elliptic boundary conditions for $\DD$. 

\begin{definition}\label{D:ellipticbc}
An \emph{elliptic boundary condition} for $\DD$ is a closed subspace $B\subset\cHH(\AAA)$ such that both $B$ and its adjoint boundary value space
\begin{equation}\label{E:adbc}
	B^\ad\;:=\;\big\{\,\bfv\in\cHH(\AAA):\,
		\big(\bfu,c(\tau)\bfv\big)=0\mbox{ for all }\bfu\in B\,\big\}
\end{equation}
are subspaces of $H^{1/2}_\AAA(\pM,E_\pM)$.
\end{definition}

We have shown in \cite{BrShi17} that an elliptic boundary value problem $\DD_B$ is Fredholm. Its \emph{index} is defined by
\[
\ind\DD_B\;:=\;\dim\ker\DD_B\,-\,\dim\ker\DD_{B^\ad}\;\in\;\ZZ.
\]

\section{Elliptic boundary conditions and Fredholm pairs}\label{S:ellbc}

In \cite{Shi18}, the second author studied the relationship between Atiyah--Patodi--Singer index and the Cauchy data spaces using the method of Fredholm pairs. In this section, we provide a new description of elliptic boundary conditions from the perspective of Fredholm pairs. The main result is similar in spirit to Definition~7.5 of \cite{BarBallmann12}.

\subsection{Fredholm pair of subspaces}\label{SS:Fredholmpair}

We recall the notion of a Fredholm pair of subspaces. Let $Z$ be a Hilbert space. A pair $(X,Y)$ of closed subspaces of $Z$ is called a \emph{Fredholm pair} if
\begin{enumerate}[label=(\roman*)]
\item $\dim(X\cap Y)<\infty$;
\item $X+Y$ is a closed subspace of $Z$;
\item $\codim(X+Y):=\dim Z/(X+Y)<\infty$.
\end{enumerate}
The \emph{index} of a Fredholm pair $(X,Y)$ is defined to be
\[
\ind(X,Y)\;:=\;\dim(X\cap Y)\,-\,\codim(X+Y)\;\in\;\ZZ.
\]

\subsection{Fredholm pairs associated to an elliptic boundary condition}\label{SS:ellipticpair}

Let $\DD:C^\infty(M,E)\to C^\infty(M,E)$ be a self-adjoint strongly Callias-type operator on a manifold with non-compact boundary. Let $\AAA:C^\infty(\dm,E_{\dm})\to C^\infty(\dm,E_{\dm})$ be the restriction of $\DD$ to $\dm$. Recall that by Definition~\ref{D:ellipticbc} a closed subspace $B$ of $\cHH(\AAA)$ is an elliptic boundary condition if $B\subset H_\AAA^{1/2}(\dm,E_{\dm})$ and $B^\ad\subset H_{\AAA}^{1/2}(\dm,E_{\dm})$. Since the $H_\AAA^{1/2}$-norm is stronger than the $\cHH$-norm, in this case $B$ is also closed in $H_\AAA^{1/2}(\dm,E_{\dm})$. Generally, if $B$ is a closed subspace of $H_\AAA^{1/2}(\dm,E_{\dm})$, then we define
\begin{equation}\label{E:Bstar}
B^*\;:=\;B^0\,\cap\,H_\AAA^{1/2}(\dm,E_{\dm}),
\end{equation}
where $B^0\subset H_\AAA^{-1/2}(\dm,E_{\dm})$ is the annihilator of $B$. The main result of this section is the following equivalent definition of elliptic boundary conditions.

\begin{theorem}\label{T:equiv-ellbc}
A subspace $B\subset H_\AAA^{1/2}(\dm,E_{\dm})$ is an elliptic boundary condition for $\DD$ if and only if
\begin{enumerate}
\item $B$ is closed in $H_\AAA^{1/2}(\dm,E_{\dm})$,
\item $(H_{[0,\infty)}^{1/2}(\AAA),B)$, $(H_{(-\infty,0)}^{1/2}(\AAA),B^*)$ are Fredholm pairs in $H_\AAA^{1/2}(\dm,E_{\dm})$, and
\item $\ind(H_{[0,\infty)}^{1/2}(\AAA),B)=-\ind(H_{(-\infty,0)}^{1/2}(\AAA),B^*)$.
\end{enumerate}
If $B$ satisfies (1), (2) and (3), then $B^\ad=c(\tau)B^*$.
\end{theorem}

A very typical elliptic boundary condition is the (generalized) Atiyah--Patodi--Singer boundary condition $B=H_{(-\infty,a)}^{1/2}(\AAA)$ for some $a\in\RR$. In this case $B^*=H_{[a,\infty)}^{1/2}(\AAA)$. One immediately sees that both $(H_{[0,\infty)}^{1/2}(\AAA),H_{(-\infty,a)}^{1/2}(\AAA))$ and $(H_{(-\infty,0)}^{1/2}(\AAA),H_{[a,\infty)}^{1/2}(\AAA))$ are Fredholm pairs in $H_\AAA^{1/2}(\dm,E_{\dm})$, and
\[
\ind(H_{[0,\infty)}^{1/2}(\AAA),H_{(-\infty,a)}^{1/2}(\AAA))=-\dim L_{[a,0)}^2(\AAA)=-\ind(H_{(-\infty,0)}^{1/2}(\AAA),H_{[a,\infty)}^{1/2}(\AAA))\quad\mbox{for }a<0
\]
or
\[
\ind(H_{[0,\infty)}^{1/2}(\AAA),H_{(-\infty,a)}^{1/2}(\AAA))=\dim L_{[0,a)}^2(\AAA)=-\ind(H_{(-\infty,0)}^{1/2}(\AAA),H_{[a,\infty)}^{1/2}(\AAA))\quad\mbox{for }a\ge0.
\]

We break the proof of Theorem~\ref{T:equiv-ellbc} into several steps which occupy the next two subsections.

\subsection{Proof of the ``if'' direction}\label{SS:if}

We apply the arguments of \cite[Subsection 3.3]{Shi18}. Suppose $(H_{[0,\infty)}^{1/2}(\AAA),B)$ is a Fredholm pair in $H_\AAA^{1/2}(\dm,E_{\dm})$. Write $B$ in the following direct sum of the  pair of transversal subspaces
\[
B\;=\;(H_{[0,\infty)}^{1/2}(\AAA)\cap B)\;\dot{+}\;V,
\]
where $V$ is some closed subspace of $H_\AAA^{1/2}(\dm,E_{\dm})$. Let $\pi_<$ (resp. $\pi_\ge$) be the projection of $V$ onto $H_{(-\infty,0)}^{1/2}(\AAA)$ (resp. $H_{[0,\infty)}^{1/2}(\AAA)$) along $H_{[0,\infty)}^{1/2}(\AAA)$ (resp. $H_{(-\infty,0)}^{1/2}(\AAA)$). Then $\pi_<$ is injective and
\[
\range\pi_<\;=\;(H_{[0,\infty)}^{1/2}(\AAA)+B)\,\cap\,H_{(-\infty,0)}^{1/2}(\AAA)
\]
is closed (in both $H_\AAA^{1/2}(\dm,E_{\dm})$ and $\cHH(\AAA)$). By closed graph theorem, $\pi_<$ has a bounded inverse $\iota_<:\range\pi_<\to V$. One then has a bounded operator $\phi:=\pi_\ge\circ\iota_<:\range\pi_<\to\range\pi_\ge$. Then $V= {\rm graph}(\phi)$, and, hence,
\[
B\;=\;(H_{[0,\infty)}^{1/2}(\AAA)\cap B)\;\dot{+}\;{\rm graph}(\phi).
\]
Let $\check{\phi}$ be the composition
\[
\range\pi_<\;\xrightarrow{\phi}\;\range\pi_\ge\;\hookrightarrow\;H_{[0,\infty)}^{-1/2}(\AAA).
\]
Viewed as a map from a closed subspace of $\cHH(\AAA)$ to $\cHH(\AAA)$, $\check{\phi}$ is a bounded operator. Note that $B$ can also be written as
\begin{equation}\label{E:Bdecomp-check}
	B\;=\;
	(H_{[0,\infty)}^{1/2}(\AAA)\cap B)
	\;\dot{+}\;
	{\rm graph}(\check{\phi})\;\subset\;\cHH(\AAA).
\end{equation}
Since the first summand is finite-dimensional, $B$ is closed in $\cHH(\AAA)$.

It now remains to show that $B^\ad\subset H_\AAA^{1/2}(\dm,E_{\dm})$. We need the following lemma.

\begin{lemma}\label{L:Fredpair-check}
$(H_{[0,\infty)}^{-1/2}(\AAA),B)$ is a Fredholm pair in $\cHH(\AAA)$ and
\[
\ind(H_{[0,\infty)}^{-1/2}(\AAA),B)\;=\;\ind(H_{[0,\infty)}^{1/2}(\AAA),B).
\]
\end{lemma}

\begin{proof}
Let $\Pi_<$ denote the projection from $\cHH(\AAA)$ onto $H_{(-\infty,0)}^{1/2}(\AAA)$ along $H_{[0,\infty)}^{1/2}(\AAA)$. Then $\Pi_<$ is an orthogonal with respect to the scalar product on $\cHH(\AAA)$. Its restriction to $H_\AAA^{1/2}(\dm,E_{\dm})$ is also the orthogonal projection (with respect to the scalar product on $H_\AAA^{1/2}(\dm,E_{\dm})$) from $H_\AAA^{1/2}(\dm,E_{\dm})$ onto $H_{(-\infty,0)}^{1/2}(\AAA)$. Since $(H_{[0,\infty)}^{1/2}(\AAA),B)$ is a Fredholm pair in $H_\AAA^{1/2}(\dm,E_{\dm})$, by \cite[Proposition 3.5]{Shi18}, $\Pi_<|_B:B\to H_{(-\infty,0)}^{1/2}(\AAA)$ is a Fredholm operator and
\[
\ind(H_{[0,\infty)}^{1/2}(\AAA),B)\;=\;\ind\Pi_<|_B.
\]
Note that $B$ is also closed in $\cHH(\AAA)$. It then follows that $(H_{[0,\infty)}^{-1/2}(\AAA),B)$ is a Fredholm pair in $\cHH(\AAA)$ and
\[
\ind(H_{[0,\infty)}^{-1/2}(\AAA),B)\;=\;\ind\Pi_<|_B\;=\;\ind(H_{[0,\infty)}^{1/2}(\AAA),B).
\]
\end{proof}

Now we use the hypothesis that $(H_{(-\infty,0)}^{1/2}(\AAA),B^*)$ is a Fredholm pair in $H_\AAA^{1/2}(\dm,E_{\dm})$. 
Applying the above discussions similarly, one concludes that $B^*$ is a closed subspace of $\hHH(\AAA)$. Moreover, $(H_{(-\infty,0)}^{-1/2}(\AAA),B^*)$ is a Fredholm pair in $\hHH(\AAA)$ satisfying
\begin{equation}\label{E:B*=B*}
\ind(H_{(-\infty,0)}^{-1/2}(\AAA),B^*)\;=\;\ind(H_{(-\infty,0)}^{1/2}(\AAA),B^*).
\end{equation}
Recall that $B$ is closed in $\cHH(\AAA)$. We denote its annihilator as a subspace of $\cHH(\AAA)$ to be
\begin{equation}\label{E:Bhat0}
	\hat{B}^0\;:=\;B^0\,\cap\,\hHH(\AAA).
\end{equation}
Clearly, $B^*\subset\hat{B}^0$. From Lemma \ref{L:Fredpair-check} and \cite[Proposition 3.5]{Shi18}, we obtain that $(H_{(-\infty,0)}^{-1/2}(\AAA),\hat{B}^0)$ is a Fredholm pair in $\hHH(\AAA)$ and
\begin{equation}\label{E:Bhat0=-B}
\ind(H_{(-\infty,0)}^{-1/2}(\AAA),\hat{B}^0)\;=\;-\ind(H_{[0,\infty)}^{-1/2}(\AAA),B)\;=\;-\ind(H_{[0,\infty)}^{1/2}(\AAA),B).
\end{equation}
Combining \eqref{E:B*=B*}, \eqref{E:Bhat0=-B} and Theorem \ref{T:equiv-ellbc}.(3) yields that
\[
\ind(H_{(-\infty,0)}^{-1/2}(\AAA),B^*)\;=\;\ind(H_{(-\infty,0)}^{-1/2}(\AAA),\hat{B}^0).
\]
Now using \cite[Lemma 4.1]{Shi18}, we finally obtain $B^*=\hat{B}^0$. Since $B^\ad=c(\tau)\hat{B}^0$, one deduces that $B^\ad\subset H_\AAA^{1/2}(\dm,E_{\dm})$. Therefore $B$ is an elliptic boundary condition for $\DD$.

\subsection{Proof of the ``only if'' direction}\label{SS:onlyif}

Let $B\subset\cHH(\AAA)$ be an elliptic boundary condition. Then condition (1) of Theorem \ref{T:equiv-ellbc} is automatically true and the $\cHH$-norm is equivalent to the $H_\AAA^{1/2}$-norm on $B$. Let $\Pi_<$ (resp. $\Pi_\ge$) denote the orthogonal projection from $\cHH(\AAA)$ onto $H_{(-\infty,0)}^{1/2}(\AAA)$ (resp. $H_{[0,\infty)}^{-1/2}(\AAA)$). For any $\bfu\in B$, there exists a constant $C>0$ such that
\[
\|\bfu\|_{H_\AAA^{1/2}}\;\le\;C\|\bfu\|_{\cHH(\AAA)}\;=\;C\big(\|\Pi_<\bfu\|_{H_\AAA^{1/2}}+\|\Pi_\ge\bfu\|_{H_\AAA^{-1/2}}\big).
\]
By \cite[Theorem 3.19]{BrShi17}, the map
\[
\Pi_\ge\;:\;B\;\subset\;H_\AAA^{1/2}(\dm,E_{\dm})\;\to\;H_{[0,\infty)}^{1/2}(\AAA)\;\hookrightarrow\;H_{[0,\infty)}^{-1/2}(\AAA)
\]
is compact. Using \cite[Proposition A.3]{BarBallmann12}, one concludes that $\Pi_<:B\to H_{(-\infty,0)}^{1/2}(\AAA)$ has finite-dimensional kernel and closed image. In other words, $H_{[0,\infty)}^{1/2}(\AAA)\cap B$ is finite-dimensional and $H_{[0,\infty)}^{1/2}(\AAA)+B$ (resp. $H_{[0,\infty)}^{-1/2}(\AAA)+B$) is a closed subspace of $H_\AAA^{1/2}(\dm,E_{\dm})$ (resp. $\cHH(\AAA)$). By the fact that $B$ is an elliptic boundary condition,
\[
B^\ad\;=\;c(\tau)\hat{B}^0\;\subset\;H_\AAA^{1/2}(\dm,E_{\dm}),
\]
where $\hat{B}^0$ is defined in \eqref{E:Bhat0}. Viewed as a subspace of $\cHH(\AAA)$, the annihilator of $H_{[0,\infty)}^{-1/2}(\AAA)+B$ is
\[
H_{(-\infty,0)}^{-1/2}(\AAA)\cap\hat{B}^0\;\cong\;H_{(0,\infty)}^{-1/2}(\AAA)\cap B^\ad\;=\;H_{(0,\infty)}^{1/2}(\AAA)\cap B^\ad.
\]
By the same reason as above,
\[
+\infty\;>\;\dim(H_{(0,\infty)}^{1/2}(\AAA)\cap B^\ad)\;=\;\codim(H_{[0,\infty)}^{-1/2}(\AAA)+B).
\]
Therefore $(H_{[0,\infty)}^{-1/2}(\AAA),B)$ is a Fredholm pair in $\cHH(\AAA)$.

Using the same argument as in the proof of Lemma \ref{L:Fredpair-check}, we can prove the following 
\begin{lemma}\label{L:Fredpair-1/2}
$(H_{[0,\infty)}^{1/2}(\AAA),B)$ is a Fredholm pair in $H_\AAA^{1/2}(\dm,E_{\dm})$ and
\begin{equation}\label{E:1/2=check}
\ind(H_{[0,\infty)}^{1/2}(\AAA),B)\;=\;\ind(H_{[0,\infty)}^{-1/2}(\AAA),B).
\end{equation}
\end{lemma}

Since $B$ is an elliptic boundary condition, from the discussion above, $\hat{B}^0\subset H_\AAA^{1/2}(\dm,E_{\dm})$. Thus
\[
B^*\;=\;\hat{B}^0\,\cap\,H_\AAA^{1/2}(\dm,E_{\dm})\;=\;\hat{B}^0.
\]
By the fact that $(H_{[0,\infty)}^{-1/2}(\AAA),B)$ is a Fredholm pair in $\cHH(\AAA)$, one concludes that $(H_{(-\infty,0)}^{-1/2}(\AAA),B^*)$ is a Fredholm pair in $\hHH(\AAA)$. It then follows from Lemma \ref{L:Fredpair-1/2} that $(H_{(-\infty,0)}^{1/2}(\AAA),B^*)$ is a Fredholm pair in $H_\AAA^{1/2}(\dm,E_{\dm})$ and
\[
\ind(H_{(-\infty,0)}^{1/2}(\AAA),B^*)\;=\;\ind(H_{(-\infty,0)}^{-1/2}(\AAA),B^*)\;=\;-\ind(H_{[0,\infty)}^{-1/2}(\AAA),B)\;=\;-\ind(H_{[0,\infty)}^{1/2}(\AAA),B),
\]
so conditions (2) and (3) of Theorem \ref{T:equiv-ellbc} are verified. We thus complete the proof of Theorem \ref{T:equiv-ellbc}. \hfill$\square$

\subsection{An index formula regarding Fredholm pairs}\label{SS:indfor-Fredpair}

Let $B_1$ and $B_2$ be two elliptic boundary conditions for $\DD$. Then $\DD_{B_1}$ and $\DD_{B_2}$ are Fredholm operators. By Theorem \ref{T:equiv-ellbc}, $(H_{[0,\infty)}^{1/2}(\AAA),B_1)$ and $(H_{[0,\infty)}^{1/2}(\AAA),B_2)$ are Fredholm pairs in $H_\AAA^{1/2}(\dm,E_{\dm})$. The following theorem, which generalizes \cite[Proposition 5.8]{BrShi17} and can be compared with \cite[Theorem 8.15]{BarBallmann12} (where the boundary is compact), computes the difference of the two elliptic boundary value problems in terms of the indexes of Fredholm pairs.

\begin{theorem}\label{T:indfor-Fredpair}
Let $B_1,B_2\subset H_\AAA^{1/2}(\dm,E_{\dm})$ be elliptic boundary conditions for $\DD$. Then
\begin{equation}\label{E:indfor-Fredpair-1}
\ind\DD_{B_1}\;-\;\ind\DD_{B_2}\;=\;\ind(H_{[0,\infty)}^{1/2}(\AAA),B_1)\;-\;\ind(H_{[0,\infty)}^{1/2}(\AAA),B_2).
\end{equation}
Let $B_2^\perp$ be the orthogonal complement of $B_2$ in $H_\AAA^{1/2}(\dm,E_{\dm})$. If $(B_2^\perp,B_1)$ is a Fredholm pair in $H_\AAA^{1/2}(\dm,E_{\dm})$, then
\begin{equation}\label{E:indfor-Fredpair-2}
\ind\DD_{B_1}\;-\;\ind\DD_{B_2}\;=\;\ind(B_2^\perp,B_1).
\end{equation}
\end{theorem}

\begin{remark}\label{R:indfor-Fredpair}
Note that the hypothesis that $(B_2^\perp,B_1)$ is a Fredholm pair is essential. In Section \ref{S:locbvp} (cf. Remark~\ref{R:notFredholmpair}), we will provide an example that the hypothesis does not hold.
\end{remark}

\begin{proof}
We first show that \eqref{E:indfor-Fredpair-1} can be implied by \eqref{E:indfor-Fredpair-2}. Since both $(H_{[0,\infty)}^{1/2}(\AAA),B_1)$ and $(H_{[0,\infty)}^{1/2}(\AAA),B_2)$ are Fredholm pairs in $H_\AAA^{1/2}(\dm,E_{\dm})$, by \eqref{E:indfor-Fredpair-2},
\[
\begin{aligned}
\ind\DD_{B_1}\;-\;\ind\DD_{\rm APS}&\;=\;\ind(H_{[0,\infty)}^{1/2}(\AAA),B_1), \\
\ind\DD_{B_2}\;-\;\ind\DD_{\rm APS}&\;=\;\ind(H_{[0,\infty)}^{1/2}(\AAA),B_2),
\end{aligned}
\]
where $\DD_{\rm APS}$ denotes the APS boundary value problem for $\DD$. Then \eqref{E:indfor-Fredpair-1} is verified by taking the difference of the two equations.

We now prove \eqref{E:indfor-Fredpair-2}. Since $(B_2^\perp,B_1)$ is a Fredholm pair in $H_\AAA^{1/2}(\dm,E_{\dm})$, we can adapt the idea of Subsection \ref{SS:if} to write $B_1$ in the following form
\[
B_1\;=\;(B_2^\perp\cap B_1)\;\dot{+}\;{\rm graph}(\phi),
\]
where $\phi$ is a bounded operator from $(B_2^\perp+B_1)\cap B_2$ to $B_2^\perp$. For $0\le s\le1$, let
\[
B_{1,s}\;=\;(B_2^\perp\cap B_1)\;\dot{+}\;{\rm graph}(s\phi).
\]
Then $B_{1,1}=B_1$ and $B_{1,s}$ is an elliptic boundary condition for each $s$. Moreover, $(B_2^\perp,B_{1,s})$ is a Fredholm pair in $H_\AAA^{1/2}(\dm,E_{\dm})$ and
\[
\ind(B_2^\perp,B_{1,s})\;=\;\ind(B_2^\perp,B_1),\qquad\mbox{for any }s\in[0,1].
\]
Consider the family of Fredholm operators $\DD_{B_{1,s}}$. Applying the arguments of \cite[Theorem 8.12]{BarBallmann12} indicates that
\[
\ind\DD_{B_{1,s}}\;=\;\ind\DD_{B_1},\qquad\mbox{for any }s\in[0,1].
\]
Thus without loss of generality, one can assume that
\[
B_1\;=\;B_{1,0}\;=\;(B_2^\perp\cap B_1)\;\oplus\;((B_2^\perp+B_1)\cap B_2).
\]

Let $X$ be the orthogonal complement of $B_2^\perp+B_1$ in $H_\AAA^{1/2}(\dm,E_{\dm})$. Note that $X$ is a finite-dimensional space, so $B_1\oplus X$ is still an elliptic boundary condition for $\DD$. Since $B_1,B_2\subset B_1\oplus X$, using the idea of \cite[Corollary 8.8]{BarBallmann12}, we have
\[
\begin{aligned}
\ind\DD_{B_1}\;-\;\ind\DD_{B_1\oplus X}&\;=\;-\dim X\;=\;-\codim(B_2^\perp+B_1), \\
\ind\DD_{B_2}\;-\;\ind\DD_{B_1\oplus X}&\;=\;-\dim(B_2^\perp\cap B_1).
\end{aligned}
\]
Taking the difference of the two equations yields \eqref{E:indfor-Fredpair-2}.
\end{proof}

\section{A local boundary value problem for strongly Callias-type operators}\label{S:locbvp}

In this section we introduce a local boundary condition for $\DD$ and show that it is elliptic.

\subsection{Splitting of the vector bundle on the boundary}\label{SS:bundlesplit}

From now on we assume that there is given an orthogonal decomposition
\begin{equation}\label{E:bundlesplitting}
	E_{\dm}\;=\;\E^+\;\oplus \E^-
\end{equation}
such that
\begin{equation}\label{E:operatorsplitting}
		\AAA\ =\
	\begin{pmatrix}
		0 & \AAA^- \\
		\AAA^+& 0
\end{pmatrix}  
\end{equation}
with respect to the grading \eqref{E:bundlesplitting}. Here $\AAA^-=(\AAA^+)^*$. 

The grading \eqref{E:bundlesplitting}  might or might not be induced by the Clifford multiplication $c(\tau)$. 
One of the main and most natural examples of such grading is the grading defined in Section~\ref{SS:grading}. This is the grading considered by Freed in \cite{Freed98}.

\subsection{Boundary value space induced by a graded bundle}\label{SS:splitbvs}

By the decomposition \eqref{E:bundlesplitting}, each $\bfu\in H_\AAA^{1/2}(\dm,E_\pM)$ has the form $\bfu=(\bfu^+,\bfu^-)$. Consider closed subspaces
\begin{equation}\label{E:localbc}
	B^+\;=\;\{(\bfu^+,0)\in H_\AAA^{1/2}(\dm,E_\pM)\},
	\qquad 
	B^-\;=\;\{(0,\bfu^-)\in H_\AAA^{1/2}(\dm,E_\pM)\}
\end{equation}
of $H_\AAA^{1/2}(\dm,E_\pM)$. Then $H_\AAA^{1/2}(\dm,E_\pM)=B^+\oplus B^-$.

\begin{proposition}\label{P:FredpairB+}
$(H_{[0,\infty)}^{1/2}(\tAAA),B^+)$ is a Fredholm pair in $H_\AAA^{1/2}(\dm,E_{\dm})$ and
\begin{equation}\label{E:FredpairB+}
\ind(H_{[0,\infty)}^{1/2}(\tAAA),B^+)\;=\;\dim\ker\tAAA^+.
\end{equation}
Similarly, $(H_{[0,\infty)}^{1/2}(\tAAA),B^-)$ is a Fredholm pair and 
\begin{equation}\label{E:FredpairB-}
	\ind(H_{[0,\infty)}^{1/2}(\tAAA),B^-)\;=\;\dim\ker\tAAA^-.
\end{equation}
\end{proposition}

The proof  of the proposition is based on the following
\begin{lemma}\label{L:B=ker}
We have
\begin{equation}\label{E:Bsub=ker}
	H_{[0,\infty)}^{1/2}(\tAAA)\,\cap\,B^+
	\ = \ \big\{\,(\bfu^+,0):\, \bfu^+\in \ker\tAAA^+\,\big\}.
\end{equation}
In particular, $H_{[0,\infty)}^{1/2}(\tAAA)\,\cap\,B^+\;\cong\;\ker\tAAA^+$.
\end{lemma}

\begin{proof}
Since $\tAAA$ is anti-diagonal with respect to the grading $E_\pM= \E^+\oplus \E^-$, one readily sees that if $\bfv=(\bfv^+,\bfv^-)$ is an eigenvector of $\tAAA$ associated with eigenvalue $\lambda$, then $\tilde{\bfv}:= (\bfv^+,-\bfv^-)$  is an eigenvector of $\tAAA$ associated with eigenvalue $-\lambda$.

Each $\bfu=(\bfu^+,0)\in H_{[0,\infty)}^{1/2}(\tAAA)\,\cap\,B^+$  has an expansion into a sum of eigenvectors  
$\bfu= \sum a_j\bfv_j$, where $a_j\in \CC$ and $\bfv_j=(\bfv_j^+,\bfv_j^-)$ is an eigenvector of $\AAA$ associated with eigenvalue $\lambda_j\ge0$. Then $\bfu^+= \sum a_j\bfv_j^+$ and $0= \sum a_j\bfv_j^-$. 
It follows that $\bfu= \sum a_j\tilde{\bfv_j}$, where as above $\tilde{\bfv_j}= (\bfv_j^+,-\bfv_j^-)$. Since $\tilde{\bfv_j}$'s are eigenvectors associated with non-positive eigenvalues we conclude that 
$\bfu\in H_{(-\infty,0]}^{1/2}(\tAAA)$. Since 
\[
	H_{(-\infty,0]}^{1/2}(\tAAA)\,\cap\,H_{[0,\infty)}^{1/2}(\tAAA)
	\;= \;\ker\tAAA,
\]
we obtain $\bfu=(\bfu^+,0)\in \ker\tAAA$. Thus $\bfu^+\in \ker\tAAA^+$. It follows that $H_{[0,\infty)}^{1/2}(\tAAA)\,\cap\,B^+\subset \big\{\,(\bfu^+,0):\, \bfu^+\in \ker\tAAA^+\,\big\}$. The opposite inclusion is obvious. 
\end{proof}

\begin{proof}[Proof of Proposition~\ref{P:FredpairB+}]
We only proof \eqref{E:FredpairB+}. The proof of the other equality is analogous. 

In view of Lemma~\ref{L:B=ker}, we only need to show that
\begin{equation}\label{E:subspacesum+}
	H_{[0,\infty)}^{1/2}(\tAAA)\,+\,B^+   \;=\;
		H_{\tAAA}^{1/2}(\dm,E_{\dm}).
\end{equation}
Choose an arbitrary $\bfu=(\bfu^+,\bfu^-)\in H_{(-\infty,0)}^{1/2}(\tAAA)$. Then $\bfv_1:=(\bfu^+,-\bfu^-)\in H_{[0,\infty)}^{1/2}(\tAAA)$. Now let $\bfv_2=(2\bfu^+,0)\in B^+$. Then $-\bfv_1+\bfv_2=\bfu$. So
\[
	H_{(-\infty,0)}^{1/2}(\tAAA)
	\;\subset\;
	H_{[0,\infty)}^{1/2}(\tAAA)\,+\,B^+.
\]
Therefore \eqref{E:subspacesum+} is true. This completes the proof of the proposition.
\end{proof}

Since $H_{[0,\infty)}^{1/2}(\tAAA)$ and $H_{(-\infty,0)}^{1/2}(\tAAA)$, $B^\pm$ and $B^\mp$ are orthogonal to each other as subspaces in $H_{\tAAA}^{1/2}(\dm,E_{\dm})$, it follows from Proposition \ref{P:FredpairB+} that

\begin{corollary}\label{C:FredpairB+}
 $(H_{(-\infty,0)}^{1/2}(\tAAA),B^\pm)$ are Fredholm pairs in $H_\AAA^{1/2}(\dm,E_{\dm})$ and 
\begin{equation}\label{E:FredpairBpm}
	\ind(H_{(-\infty,0)}^{1/2}(\tAAA),B^\pm)\;=\;-\dim\ker\tAAA^\mp.
\end{equation}
\end{corollary}

From Proposition \ref{P:FredpairB+} and Corollary \ref{C:FredpairB+}, one readily sees that $B^\pm$ satisfies Theorem \ref{T:equiv-ellbc}. Therefore we get

\begin{theorem}\label{T:ellbc}
$B^\pm$ is an elliptic boundary condition for $\DD$, whose adjoint boundary condition is $c(\tau)B^\mp$.
\end{theorem}

We call $B^\pm$ the \emph{boundary condition subject to the grading \eqref{E:bundlesplitting}}. It is a local boundary condition for $\DD$.

\begin{remark}\label{R:notFredholmpair}
If in this situation we let $B^+$ and $B^-$ be the $B_1$ and $B_2$ as in Theorem \ref{T:indfor-Fredpair}, then $(B_2^\perp,B_1)=(B^+,B^+)$ is not a Fredholm pair (cf. Remark \ref{R:indfor-Fredpair}).
\end{remark}

\begin{corollary}\label{C:ellbc}
$\DD_{B^+}$ (resp. $\DD_{B^-}$) is a Fredholm operator, whose adjoint operator is $\DD_{c(\tau)B^-}$ (resp. $\DD_{c(\tau)B^+}$).
\end{corollary}

\subsection{An index theorem}\label{SS:indextheorem}

Substitution $B_1=B^+$ and $B_2=B^-$ in Theorem~\ref{T:indfor-Fredpair} and using Proposition \ref{P:FredpairB+}, we get the following index formula for the local boundary value problem: 

\begin{theorem}\label{T:inddiff}
Let $B^\pm$ be as in Subsections~\ref{SS:splitbvs}. Then
\begin{equation}\label{E:inddiff}
	\ind\DD_{B^+}\;-\;\ind\DD_{B^-}\;=\;\ind\tAAA.
\end{equation}
\end{theorem}

\section{Index of the local boundary problem subject to natural gradings}\label{S:original}

In this section we formulate our main result -- the index theorem for a local boundary value problem similar to the one considered in \cite{Freed98}. First we obtain a vanishing result for the index subject to the grading $E_\pM=E_\pM^+\oplus E_\pM^-$ given by the action of $c(\tau)$, cf. Subsection~\ref{SS:grading}. As an application we obtain a new proof of the cobordism invariance of the index of a Callias-type operator. Then we assume that the boundary has multiple components $\pM=\bigsqcup N_j$ with the local boundary condition $u|_{N_j}\in E_\pM^\pm|_{N_j}$ and  obtain an extension of the index theorem of Freed, \cite{Freed98}, to our non-compact situation.

\subsection{A vanishing result}\label{SS:vanishing}

Choose the grading \eqref{E:bundlesplitting} so that $c(\tau)|_{E_{\dm}^\pm}=\pm i$ (this is, for example, the case for the spinor bundle on the boundary of an odd-dimensional manifold). Then near the boundary $\DD$ has the form
\begin{equation}\label{E:Calliasdecomp2}
	\DD\;=\;
	\begin{pmatrix}
		i & 0 \\
		0 & -i
	\end{pmatrix}\,
	\left(  \p_t+\begin{pmatrix}
					0 & \AAA^- \\
					\AAA^+ & 0
				\end{pmatrix}
	\right).
\end{equation}
In this case $c(\tau)B^\pm=B^\pm$, thus $\DD_{B^+}$ and $\DD_{B^-}$ are adjoint operator to each other by Corollary \ref{C:ellbc}. 

\begin{proposition}\label{P:vanishing}
Under the above assumption, $\ind\DD_{B^\pm}=0$.
\end{proposition}

\begin{proof}
A verbatim repetition of the arguments in the proof of \cite[Theorem~21.5]{BW93} shows that $\ker\DD_{B^\pm}= \{0\}$. Since $(\DD_{B^+})^{\ad}= \DD_{B^-}$, the index of $\DD_{B^\pm}$ vanishes. 
\end{proof}

Combining this proposition with Theorem~\ref{T:inddiff}, we obtain the following cobordism invariance of the index of strongly Callias-type operators (cf. \cite{BrShi16} where this result is proven by a different method. Yet another proof is given in \cite[\S2.7]{BrCecchini17}):

\begin{corollary}\label{C:vanishing}
Let $\AAA:C^\infty(N,E^\pm)\to C^\infty(N,E^\mp)$ be a graded strongly Callias-type operator on a non-compact manifold $N$. Suppose there exist a non-compact manifold $M$, a Dirac bundle $\hat{E}\to M$ and a self-adjoint strongly Callias-type operator $\DD:C^\infty(M,\hat{E})\to C^\infty(M,\hat{E})$ such that $\dm=N$, $\hat{E}|_{\dm}=E^+\oplus E^-$ and $\DD$ takes the form \eqref{E:Calliasdecomp2} near $\dm$. Then $\ind\AAA=0$.
\end{corollary}

\subsection{The case of multiple boundary components}\label{SS:multbdry}

Assume that $\dm=\bigsqcup_{j=1}^kN_j$ is a disjoint union of finitely many connected components. The restriction $E_{N_j}$ of $E$ to each connected component $N_j$ inherits the grading \eqref{E:bundlesplitting}:
\[
	E_{N_j}\;=\;E_{N_j}^+\oplus E_{N_j}^-.
\]
We denote the restriction of $\DD$ to $E_{N_j}^\pm$ by $\AAA_j^\pm$. Then $E_{\dm}^\pm=\bigoplus_{j=1}^kE_{N_j}^\pm$ and $\AAA^\pm=\bigoplus_{j=1}^k\AAA_j^\pm$. Let $\epsilon=(\epsilon_1,\dots,\epsilon_k)$ with $\epsilon_j=+$ or $-$.

\begin{definition}\label{D:mixedbc}
We call $B^\epsilon:=\bigoplus_{j=1}^kB_j^{\epsilon_j}$ the \emph{mixed boundary condition} subject to the grading \eqref{E:bundlesplitting}, where $B_j^{\epsilon_j}$ is the local boundary condition \eqref{E:localbc} on each component $(N_j,E_{N_j})$ of the boundary.
\end{definition}

Using the same arguments as in Section \ref{S:locbvp}, one can show that $B^\epsilon$ is an elliptic boundary condition for $\DD$, whose adjoint boundary condition is $B^{\bar{\epsilon}}$, where $\bar{\epsilon}_j=-\epsilon_j$.  Applying Theorem \ref{T:inddiff} to this situation we obtain the following generalization of \cite[Theorem~B]{Freed98}:

\begin{theorem}\label{T:mixindfor}
Let $M$ be a complete manifold with boundary and let $E\to M$ be a Dirac bundle over $M$. Let $\DD= D+\Psi$ be a formally self-adjoint strongly Callias-type operator on $E$. Assume that the boundary $\pM$ of $M$ is a disjoint union of finitely many connected components $\pM= \bigsqcup_{j=1}^kN_j$. Fix $\epsilon:= (\epsilon_1,\ldots,\epsilon_k)$ with $\epsilon_j=\pm$ and set $B^\epsilon= \bigoplus_{j=1}^kB_j^{\epsilon_j}$. Then
\begin{equation}\label{E:mixindfor}
\ind\DD_{B^\epsilon}\;=\;\sum_{\{j\,:\,\epsilon_j=+\}}\ind\AAA_j\;=\;-\sum_{\{j\,:\,\epsilon_j=-\}}\ind\AAA_j.
\end{equation}
\end{theorem}

\begin{proof}
The second equlity of \eqref{E:mixindfor} follows from Corollary \ref{C:vanishing}.

We can apply Proposition~\ref{P:FredpairB+}  to each boundary component to conclude that $(H_{[0,\infty)}^{1/2}(\AAA_j),B_j^{\epsilon_j})$ is a Fredholm pair in $H_{\AAA_j}^{1/2}(N_j,E_{N_j})$ and
\[
\ind(H_{[0,\infty)}^{1/2}(\AAA_j),B_j^{\epsilon_j})\;=\;\dim\ker\AAA_j^{\epsilon_j}.
\]
As in the proof of Theorem \ref{T:inddiff}, using \eqref{E:indfor-Fredpair-1} we obtain
\[
	\ind\DD_{B^\epsilon}-\ind\DD_{\rm APS}
	\;=\;
	\sum_{j=1}^k\ind(H_{[0,\infty)}^{1/2}(\AAA_j),B_j^{\epsilon_j})
	\;=\;
	\sum_{j=1}^k\dim\ker\AAA_j^{\epsilon_j}
\]
and
\[
	\ind\DD_{B^-}-\ind\DD_{\rm APS}
	\;=\;
	\sum_{j=1}^k\ind(H_{[0,\infty)}^{1/2}(\AAA_j),B_j^-)
	\;=\;
	\sum_{j=1}^k\dim\ker\AAA_j^-.
\]
By Proposition \ref{P:vanishing}, $\ind\DD_{B^-}=0$. Hence, 
\[
\begin{aligned}
	\ind\DD_{B^\epsilon}
	\;&=\;
	(\ind\DD_{B^\epsilon}-\ind\DD_{\rm APS})
		\,-\,(\ind\DD_{B^-}-\ind\DD_{\rm APS}) \\
	&=\;
	\sum_{j=1}^k(\dim\ker\AAA_j^{\epsilon_j}-\dim\ker\AAA_j^-) \
	=\;\sum_{\{j\,:\,\epsilon_j=+\}}\ind\AAA_j.
\end{aligned}
\]
This completes the proof.
\end{proof}


\end{document}